\documentclass[11pt]{article}

\usepackage[utf8]{inputenc}
\usepackage[T1]{fontenc}
\usepackage{amsmath,amssymb,amsthm}
\usepackage{geometry}
\usepackage{hyperref}
\usepackage{enumitem}
\usepackage{float}
\usepackage{tikz}
\usetikzlibrary{arrows.meta,positioning,calc,fit,backgrounds}

\newcommand{\doi}[1]{\href{https://doi.org/#1}{doi:#1}}

\hypersetup{
  colorlinks=true,
  linkcolor=blue,
  citecolor=blue,
  urlcolor=blue
}

\tikzset{
  state/.style={circle,draw,line width=0.8pt,minimum size=7mm,inner sep=1pt},
  recurrent state/.style={circle,draw,line width=0.8pt,minimum size=7mm,inner sep=1pt},
  graph edge/.style={-{Stealth[length=2mm]},line width=0.8pt},
  two way/.style={<->,>=Stealth,line width=0.8pt},
  simplex point/.style={circle,fill=black,inner sep=1.5pt},
  unrealized point/.style={circle,draw,line width=0.8pt,inner sep=1.4pt}
}

\newtheorem{proposition}{Proposition}[section]
\newtheorem{definition}[proposition]{Definition}
\newtheorem{problem}[proposition]{Problem}
\newtheorem{remark}[proposition]{Remark}

\title{Occupation Patterns and Parikh Images in Markov Support Dynamics}

\author{
Luis Pousa\\
Faculty of Engineering and Business Technology\\
Universidad Intercontinental de la Empresa (UIE)\\
A Coruña, Spain\\
\texttt{luis.pousa@uie.edu}\\
ORCID: 0009-0008-8217-5555
}

\date{}

\begin{document}
\maketitle

\begin{abstract}
The directed support graph of a discrete-time, time-homogeneous Markov chain naturally defines a regular language whose words are the admissible trajectories from a fixed initial state. Applying the Parikh map associates with each trajectory its occupation vector, recording how many times each state is visited while disregarding the chronological order of the visits. Equivalently, each trajectory determines a monomial whose exponents are its occupation numbers, yielding a natural commutative representation of occupation patterns.

For each trajectory length \(n\), we define the occupation ideal generated by the corresponding Parikh monomials. The minimal generators of this ideal are in one-to-one correspondence with the distinct occupation patterns realized at that length, providing an algebraic encoding of the combinatorial structure of admissible trajectories. This construction naturally gives rise to three complementary measures of support complexity describing reachability growth, trajectory growth, and occupation-pattern growth.

The resulting framework connects Markov support graphs, regular languages, Parikh images, and monomial ideals, establishing occupation ideals as a new algebraic object associated with Markov support dynamics. It provides a natural interface between stochastic processes, formal language theory, and combinatorial commutative algebra, and suggests new algebraic and geometric approaches to the study of support dynamics.
\end{abstract}



\section{Introduction}
\label{sec:introduction}

Discrete-time Markov chains are among the central stochastic models in probability theory and its applications \cite{levin2017,meyn2009,norris1997}. A time-homogeneous Markov chain on a finite state space $S$ is determined by a transition matrix $Q=(q_{ij})$, where
\[
q_{ij}=\mathbb{P}(X_{n+1}=j\mid X_n=i)
\]
is independent of $n$. The powers $Q^n$ describe the evolution of the process through its transition probabilities and provide the standard probabilistic description of the dynamics.

In this paper we focus on the combinatorial structure underlying these dynamics. This structure is determined by the directed support graph of the chain. Associated with $Q$ is the support adjacency matrix $B=(b_{ij})$, defined by
\[
b_{ij}=
\begin{cases}
1,& q_{ij}>0,\\
0,& q_{ij}=0.
\end{cases}
\]
The matrix $Q$ records probabilistic weights, whereas $B$ records which transitions are possible. Thus, the support graph retains the full structure of admissible trajectories while disregarding their probabilities.

The starting point of the paper is that the directed support graph naturally defines a language of admissible trajectories. Fix an initial state $i$. Every admissible trajectory
\[
\gamma=(i=i_0,i_1,\ldots,i_n)
\]
may be regarded as the word
\[
w_\gamma=i_0i_1\cdots i_n
\]
over the alphabet $S$. For finite state spaces, the set of all finite trajectories from a fixed initial state forms a regular language. Indeed, it is recognized by the finite automaton naturally associated with the support graph \cite{eilenberg1974,hopcroft2006}. Although the main definitions and results are stated for finite Markov chains, the same viewpoint applies, at each fixed time level, to locally finite support graphs.

We study the occupation profiles realized by these admissible trajectories. Applying the Parikh map \cite{parikh1966} to the support language records, for each trajectory, how many times each state is visited while forgetting the chronological order of the visits. Consequently, the Parikh image of the support language is precisely the set of occupation vectors realizable from the chosen initial state.

Equivalently, if $S=\{1,\ldots,m\}$, every admissible trajectory determines the monomial
\[
m_\gamma=x_1^{\alpha_1}\cdots x_m^{\alpha_m},
\]
where $\alpha_s$ is the number of visits to state $s$ along $\gamma$. This monomial is the commutative image of the corresponding trajectory word. It preserves the occupation data encoded by the corresponding Parikh vector while identifying trajectories that differ only in the temporal order of their visits.

This leads to the main algebraic construction of the paper. Fix an initial state $i$ and an integer $n\geq0$. We define the occupation ideal
\[
J_n(i),
\]
whose minimal generators are in one-to-one correspondence with the occupation patterns realized by admissible trajectories of length $n$ starting at $i$. Thus, the occupation ideal provides a commutative-algebraic representation of the occupation patterns at time $n$.

The construction distinguishes three complementary measures of support complexity. The reachability growth function $g_i(n)$ records the spatial expansion of the support graph. The trajectory growth function $p_i(n)$ counts admissible ordered trajectories. The occupation-pattern growth function $a_i(n)$ counts the distinct occupation vectors that remain after chronological information has been discarded.

These quantities capture different aspects of the same support dynamics. A graph may have slow reachability growth while admitting exponentially many trajectories. Conversely, many trajectories may collapse to the same occupation vector. The passage from trajectories to occupation vectors therefore provides a natural compression of temporal information, preserving state occupations while forgetting their chronological order.

Several motivations support this approach. Occupation measures play a fundamental role in the study of Markov chains \cite{levin2017,pitman1977}. Here we study their support-level combinatorial analogue, namely the occupation vectors realizable independently of their probabilities. Since admissible trajectories form a regular language, their Parikh image is semilinear by Parikh's theorem \cite{esparza2011,parikh1966}. This places occupation patterns within a classical geometric and combinatorial framework.

The monomial representation brings this structure into the setting of combinatorial commutative algebra. Unlike classical path ideals, whose generators encode ordered paths, the occupation ideals introduced in Section~\ref{sec:occupation-ideals} are generated by Parikh vectors and therefore identify trajectories with the same occupation profile \cite{he2010,kubitzke2014}. This provides an algebraic representation of support-level occupation data and makes it possible to formulate questions about Hilbert functions, Newton polyhedra, Betti numbers and related invariants in this context.

The paper is organized as follows. Section~\ref{sec:support-graphs} introduces support graphs and reachability ideals. Section~\ref{sec:parikh} develops Markov support languages and their Parikh images. Section~\ref{sec:occupation-ideals} defines occupation ideals and interprets their minimal generators as occupation patterns. Section~\ref{sec:compression} establishes elementary compression and growth properties, including the rationality of the occupation-pattern generating function. Section~\ref{sec:examples} illustrates the construction through finite irreducible chains, reducible chains and random walks. Finally, Section~\ref{sec:perspectives} discusses the geometric and algebraic interpretation of the construction and outlines several directions for future work.

The main contribution of the paper is the introduction of occupation patterns as a combinatorial object associated with Markov support dynamics, together with their representation through Parikh images and monomial ideals. This perspective connects Markov chains, formal language theory and combinatorial commutative algebra through the common structure of admissible trajectories.

\section{Support Graphs and Reachability}
\label{sec:support-graphs}

Let $Q=(q_{ij})$ be the transition matrix of a discrete-time, time-homogeneous Markov chain on a finite state space $S$. We associate with $Q$ its support adjacency matrix $B=(b_{ij})$, defined by
\[
b_{ij}=
\begin{cases}
1,& q_{ij}>0,\\
0,& q_{ij}=0.
\end{cases}
\]
The corresponding directed support graph is
\[
G_B=(S,E),
\]
where
\[
E=\{(i,j)\in S\times S:\, b_{ij}=1\}.
\]
Thus, $(i,j)\in E$ if and only if the transition $i\to j$ is admissible. The graph $G_B$ records the combinatorial structure of the chain by retaining the possible transitions and disregarding their probabilities \cite{godsil2001}.

Fix an initial state $i\in S$. For each $n\geq0$, define
\[
A_n(i)=\{j\in S:\exists\,k\leq n\text{ such that }(B^k)_{ij}>0\}.
\]
The set $A_n(i)$ consists of the states reachable from $i$ in at most $n$ steps. Equivalently, it is the ball of radius $n$ around $i$ in the directed reachability structure determined by $G_B$.

To encode this expansion algebraically, let
\[
R=k[x_s:s\in S]
\]
be the polynomial ring over a field $k$ with one variable associated with each state. We define the reachability ideal
\[
I_n(i)=\langle x_j:j\in A_n(i)\rangle\subseteq R.
\]

Since
\[
A_n(i)\subseteq A_{n+1}(i),
\]
the reachability ideals form an ascending chain
\[
I_0(i)\subseteq I_1(i)\subseteq I_2(i)\subseteq\cdots,
\]
which eventually stabilizes because the state space is finite. This filtration records the progressive expansion of the accessible region of the support graph.

The reachability growth function is
\[
g_i(n)=|A_n(i)|.
\]
It measures how many states can be reached from the initial state within the prescribed time horizon. This is the first, and simplest, support-level invariant considered in the paper; trajectory counts and occupation-pattern counts will refine this description in the following sections.

\section{Markov Support Languages and Parikh Images}
\label{sec:parikh}

The directed support graph defines a formal language of admissible trajectories. Every trajectory
\[
\gamma=(i=i_0,i_1,\ldots,i_n)
\]
starting from a fixed initial state $i$ may be regarded as the word
\[
w_\gamma=i_0i_1\cdots i_n
\]
over the alphabet $S$.

For each $n\geq0$, let $L_i^{(n)}$ denote the set of admissible trajectories of length $n$ starting at $i$, viewed as words. We define the corresponding Markov support language by
\[
L_i=\bigcup_{n\geq0}L_i^{(n)}.
\]

Since $G_B$ is finite, the language $L_i$ is regular and is recognized by the nondeterministic finite automaton induced by the support graph \cite{eilenberg1974,hopcroft2006}. The number of admissible trajectories of length $n$ starting at $i$ will be denoted by
\[
p_i(n).
\]
Equivalently,
\[
p_i(n)=\sum_{j\in S}(B^n)_{ij}.
\]
Thus, $p_i(n)$ measures the growth of the ordered trajectory space.

We now pass from ordered trajectories to occupation data. To each word $w\in L_i$ we associate its Parikh vector
\[
\Psi(w)=(\alpha_1,\ldots,\alpha_m),
\]
where $\alpha_s=|w|_s$ is the number of occurrences of the state $s$ in $w$. The Parikh map forgets the order of the symbols and retains only their multiplicities. In the present setting, $\Psi(w_\gamma)$ records exactly how many times each state is visited along the trajectory $\gamma$.

The Parikh image of the support language is
\[
P_i=\Psi(L_i).
\]
It is precisely the set of occupation vectors realized by admissible trajectories starting from $i$.

The finite-time components
\[
P_i^{(n)}=\Psi(L_i^{(n)})
\]
will play a central role below. They consist of the occupation vectors realized by trajectories of length $n$. These finite sets form the basic combinatorial objects encoded by the occupation ideals introduced in Section~\ref{sec:occupation-ideals}.

\begin{proposition}
\label{prop:parikh-semilinear}

The set
\[
P_i=\{\Psi(w):w\in L_i\}\subseteq\mathbb{N}^{|S|}
\]
is semilinear.
\end{proposition}

\begin{proof}
The language $L_i$ is regular, since it is recognized by the nondeterministic finite automaton induced by the support graph. Therefore, Parikh's theorem implies that its Parikh image is semilinear \cite{esparza2011,parikh1966}. Hence $P_i$ is a semilinear subset of $\mathbb{N}^{|S|}$.
\end{proof}

\section{Occupation Ideals}
\label{sec:occupation-ideals}

Let
\[
\gamma=(i=i_0,i_1,\ldots,i_n)
\]
be an admissible trajectory. Its occupation vector is
\[
\alpha(\gamma)=(\alpha_s)_{s\in S},
\]
where
\[
\alpha_s=\#\{r:i_r=s\}.
\]
Thus, $\alpha(\gamma)$ records how many times each state is visited along the trajectory while disregarding the order in which those visits occur.

Equivalently, we associate with $\gamma$ the monomial
\[
m_\gamma=\prod_{s\in S}x_s^{\alpha_s}
=x^{\Psi(w_\gamma)}.
\]
This monomial is the commutative image of the trajectory word under the Parikh map and provides an algebraic representation of its occupation vector.

Although the occupation vector and the associated monomial contain exactly the same information, the monomial representation offers two conceptual advantages. First, it provides a sparse representation of occupation data, since variables corresponding to states that are never visited simply do not appear. Second, it embeds occupation patterns naturally into the polynomial ring, where algebraic operations such as multiplication and divisibility, as well as the theory of monomial ideals, become immediately available. Thus, no information is lost in passing from occupation vectors to monomials, while the same combinatorial data are placed within a substantially richer algebraic framework.

For each $n\geq1$, we define the corresponding occupation ideal
\[
J_n(i)=\langle m_\gamma:\gamma\in L_i^{(n)}\rangle.
\]
Thus, $J_n(i)$ is generated by the monomials associated with all admissible trajectories of length $n$ starting from the initial state $i$. Distinct trajectories with the same occupation pattern contribute the same generator, so the ideal encodes the occupation patterns realized at time $n$.

The occupation-pattern growth function is defined by
\[
a_i(n)=\#\{\Psi(w):w\in L_i^{(n)}\}.
\]
Equivalently, $a_i(n)$ counts the number of distinct occupation vectors realized by admissible trajectories of length $n$.

\begin{proposition}
\label{prop:occupation-generators}
All minimal generators of $J_n(i)$ have degree $n+1$. Consequently,
\[
a_i(n)=\mu(J_n(i)),
\]
where $\mu(J_n(i))$ denotes the number of minimal monomial generators of $J_n(i)$.
\end{proposition}

\begin{proof}
Every admissible trajectory of length $n$ visits exactly $n+1$ states, counting multiplicities. Hence every generator $m_\gamma$ has total degree $n+1$. Since distinct monomials of the same degree cannot properly divide one another, the minimal generating set of $J_n(i)$ consists precisely of the distinct monomials arising from admissible trajectories. These monomials are in one-to-one correspondence with the distinct occupation vectors realized by trajectories in $L_i^{(n)}$. Therefore,
\[
a_i(n)=\mu(J_n(i)).
\]
\end{proof}

Together with the reachability growth function $g_i(n)$ and the trajectory growth function $p_i(n)$ introduced in the previous sections, the function $a_i(n)$ completes the three measures of support complexity considered in this paper:
\[
g_i(n),\qquad
p_i(n),\qquad
a_i(n),
\]
corresponding respectively to reachability growth, trajectory growth and occupation-pattern growth.

The passage
\[
p_i(n)\longrightarrow a_i(n)
\]
may be interpreted as a compression of trajectory information. Distinct admissible trajectories may determine the same occupation vector and therefore the same monomial generator. The occupation ideals record precisely the information that survives after the temporal order of the trajectories has been discarded.

\section{Parikh Compression}
\label{sec:compression}

The occupation ideals introduced in the previous section naturally raise two questions. How much information is lost when trajectories are replaced by their occupation vectors? And how does the number of occupation patterns grow as the trajectory length increases? This section addresses these questions.

Since every occupation pattern is realized by at least one admissible trajectory,
\[
a_i(n)\leq p_i(n).
\]
The inequality reflects the fact that the Parikh map forgets chronological order: distinct trajectories may determine the same occupation vector. We refer to this phenomenon as \emph{Parikh compression} \cite{parikh1966,to2010}.

\begin{definition}
\label{def:parikh-compression}
Whenever $p_i(n)>0$, the \emph{Parikh compression ratio} at time $n$ is defined by
\[
C_i(n)=\frac{a_i(n)}{p_i(n)}.
\]
\end{definition}

The compression ratio $C_i(n)$ quantifies the amount of combinatorial information preserved under the Parikh map. Equivalently, it quantifies the loss of chronological information when admissible trajectories are identified through their occupation patterns.

The following remark clarifies the role of directed cycles in Parikh compression.

\begin{remark}
\label{rem:acyclic-compression}
If the support graph is acyclic, every admissible trajectory visits each state at most once. Hence the Parikh map is injective on admissible trajectories of fixed length, and therefore
\[
a_i(n)=p_i(n)
\]
for every $n\ge1$.

Thus, directed cycles are necessary for nontrivial Parikh compression. They are not sufficient, however. For example, consider the support graph
\[
1\longrightarrow2,\qquad
2\longrightarrow3,\qquad
3\longrightarrow2.
\]
Starting from state $1$, there is exactly one admissible trajectory of each length. Hence
\[
a_1(n)=p_1(n)
\]
for every $n$, despite the presence of the directed cycle
\[
2\longrightarrow3\longrightarrow2.
\]
Compression therefore depends not only on the existence of directed cycles, but also on the combinatorial structure of admissible trajectories.
\end{remark}

The previous remark identifies a qualitative condition for the absence of Parikh compression. The next remark complements it with a simple quantitative bound on trajectory and occupation-pattern growth.

\begin{remark}
\label{rem:degree-bound}
If every state of the support graph has out-degree at most $d$, then
\[
a_i(n)\leq p_i(n)\leq d^n
\]
for every $n\geq0$.

Each step of an admissible trajectory offers at most $d$ possible continuations, while every occupation pattern is realized by at least one trajectory.
\end{remark}

This simple bound already shows that trajectory growth and occupation-pattern growth may differ substantially. Even when the number of admissible trajectories grows exponentially, many of them may collapse to the same occupation pattern under the Parikh map.

\subsection{Generating Function of Occupation-Pattern Growth}
\label{subsec:generating-function}

The preceding results show that occupation-pattern growth is constrained by trajectory growth. A stronger structural property follows from the semilinearity of the Parikh image.

\begin{proposition}
\label{prop:rational-generating-function}
Let $G_B$ be a finite support graph and fix an initial state $i$. Then the generating function
\[
A_i(t)=\sum_{n\geq0}a_i(n)t^n
\]
is rational.
\end{proposition}

\begin{proof}
By Proposition~\ref{prop:parikh-semilinear}, the Parikh image
\[
P_i=\Psi\bigl(L_i\bigr)
\]
is semilinear. A classical result of Eilenberg and Sch\"utzenberger \cite{eilenberg-schutzenberger1969}, together with the exposition in \cite{stanley2012}, implies that the multivariate generating function
\[
F_i(x_1,\ldots,x_m)
=
\sum_{\alpha\in P_i}x^\alpha
\]
is rational.

For each $n\geq0$, let
\[
P_i^{(n)}=\Psi\bigl(L_i^{(n)}\bigr).
\]
Every vector
\[
\alpha=(\alpha_1,\ldots,\alpha_m)\in P_i^{(n)}
\]
satisfies
\[
\alpha_1+\cdots+\alpha_m=n+1,
\]
since every trajectory of length $n$ visits exactly $n+1$ states.

After specializing
\[
x_1=\cdots=x_m=t,
\]
each occupation vector in $P_i^{(n)}$ contributes the monomial $t^{n+1}$. Since there are exactly
\[
a_i(n)=|P_i^{(n)}|
\]
such vectors,
\[
F_i(t,\ldots,t)
=
\sum_{n\geq0}a_i(n)t^{n+1}
=
tA_i(t).
\]
Hence
\[
A_i(t)=t^{-1}F_i(t,\ldots,t),
\]
which is rational because $F_i$ is rational.
\end{proof}

The rationality of $A_i(t)$ shows that occupation-pattern growth is far from arbitrary. In particular, the sequence $a_i(n)$ eventually satisfies a linear recurrence with constant coefficients. Determining whether its growth is polynomial, quasi-polynomial or exponential requires a finer analysis of the poles of $A_i(t)$. Their relation to the combinatorial structure of the support graph remains an interesting direction for future investigation.

\begin{problem}[Asymptotic Classification of Occupation-Pattern Growth]
\label{prob:asymptotic-classification}

Classify the asymptotic behaviour of
\[
a_i(n)=\#\Psi\bigl(L_i^{(n)}\bigr)
\]
in terms of the combinatorial structure of the support graph.

More specifically, determine how the growth of $a_i(n)$ depends on the branching structure, the strongly connected components and the directed cycle structure of $G_B$. In particular:

\begin{enumerate}[label=\arabic*.]

\item For which classes of support graphs does
\[
a_i(n)\sim p_i(n)
\]
hold?

\item For which support graphs is $a_i(n)$ eventually polynomial or quasi-polynomial?

\item For which support graphs is $a_i(n)$ exponential?

\item How are the dominant poles of $A_i(t)$ related to the cycle geometry of the support graph?

\end{enumerate}

\end{problem}

These questions arise naturally from the constructions developed in this paper. Their resolution is expected to depend on the interaction between branching, recurrence and directed cycles, which together govern the geometry of the Parikh image and the asymptotic behaviour of occupation-pattern growth. Understanding these interactions appears to be a natural next step in the development of the theory.

\section{Illustrative Examples}
\label{sec:examples}

The following examples illustrate how occupation patterns capture structural features of Markov support dynamics that are not reflected by reachability or trajectory counts alone. Together, they show how the three complexity functions
\[
g_i(n),\qquad
p_i(n),\qquad
a_i(n),
\]
capture complementary aspects of the support dynamics.

\subsection{A Finite Irreducible Support Graph}
\label{subsec:irreducible-example}

Our first example illustrates the simplest occurrence of Parikh compression. Even for a small irreducible support graph, distinct admissible trajectories may determine the same occupation pattern, showing that occupation-pattern growth and trajectory growth need not coincide.

Consider the support graph whose adjacency matrix is
\[
B=
\begin{pmatrix}
1&1&1&0\\
0&0&1&1\\
0&1&0&0\\
1&0&0&1
\end{pmatrix}.
\]

Starting from state $1$,
\[
A_0(1)=\{1\},\qquad
A_1(1)=\{1,2,3\},\qquad
A_2(1)=\{1,2,3,4\},
\]
so that the reachability process stabilizes after two steps.

The corresponding support graph is shown in Figure~\ref{fig:irreducible-support}.

\begin{figure}[ht]
\centering
\begin{tikzpicture}

\node[state] (1) at (0,2) {$1$};
\node[state] (2) at (3,2) {$2$};
\node[state] (3) at (3,0) {$3$};
\node[state] (4) at (0,0) {$4$};

\draw[graph edge,loop above] (1) to (1);
\draw[graph edge,loop below] (4) to (4);

\draw[graph edge] (1) -- (2);
\draw[graph edge] (1) -- (3);
\draw[graph edge,bend left=18] (2) to (3);
\draw[graph edge,bend left=18] (3) to (2);
\draw[graph edge,bend left=25] (2) to (4);
\draw[graph edge] (4) -- (1);

\end{tikzpicture}
\caption{Support graph of the four-state irreducible example.}
\label{fig:irreducible-support}
\end{figure}
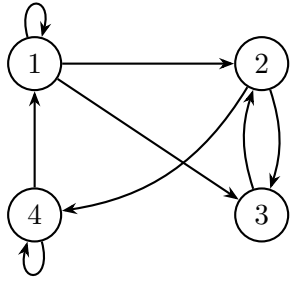

At length two there are six admissible trajectories,
\[
1\to1\to1,\quad
1\to1\to2,\quad
1\to1\to3,\quad
1\to2\to3,\quad
1\to2\to4,\quad
1\to3\to2,
\]
whose occupation vectors are
\[
(3,0,0,0),\;
(2,1,0,0),\;
(2,0,1,0),\;
(1,1,1,0),\;
(1,1,0,1),\;
(1,1,1,0).
\]

The trajectories
\[
1\to2\to3
\quad\text{and}\quad
1\to3\to2
\]
are distinct but have the same occupation pattern. Hence
\[
g_1(2)=4,\qquad
p_1(2)=6,\qquad
a_1(2)=5.
\]

This example already exhibits the three levels of support complexity introduced in the paper. Reachability stabilizes after two steps, whereas both trajectory complexity and occupation-pattern complexity continue to evolve. The inequality
\[
a_1(2)<p_1(2)
\]
is the first manifestation of Parikh compression.

\subsection{Reducible Chains}
\label{subsec:reducible-example}

Our second example illustrates how occupation patterns reflect the classical decomposition of a Markov chain into transient and recurrent classes.

Consider the reducible support graph shown in Figure~\ref{fig:reducible-support}. States $1$ and $2$ are transient, while $\{3,4\}$ forms a closed recurrent class.

\begin{figure}[ht]
\centering
\begin{tikzpicture}[node distance=2.2cm]
\node[state] (1) {$1$};
\node[state,right of=1] (2) {$2$};
\node[recurrent state,right of=2] (3) {$3$};
\node[recurrent state,right of=3] (4) {$4$};

\draw[graph edge,loop above] (1) to (1);
\draw[graph edge] (1) to (2);
\draw[graph edge,bend left=18] (1) to (3);
\draw[graph edge] (2) to (3);
\draw[graph edge,loop above] (3) to (3);
\draw[two way] (3) to (4);
\draw[graph edge,loop above] (4) to (4);

\end{tikzpicture}
\caption{Support graph of the reducible chain.}
\label{fig:reducible-support}
\end{figure}
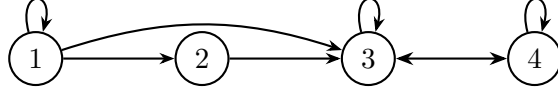

Starting from state $1$, every trajectory that enters the recurrent class $\{3,4\}$ remains there forever, since the class is closed. Consequently, every sufficiently long admissible trajectory is ultimately governed by the recurrent dynamics. From the language-theoretic viewpoint, the corresponding support language naturally decomposes into a transient prefix followed by a recurrent suffix.

Typical long admissible trajectories include
\[
1\to2\to3\to4\to3\to3\to4\to\cdots\to3\to4,
\]
with occupation vector
\[
(1,1,27,31),
\]
and
\[
1\to1\to3\to4\to4\to3\to4\to\cdots\to4\to3,
\]
with occupation vector
\[
(2,0,29,28).
\]

These examples illustrate the characteristic asymptotic behaviour of reducible support graphs. Once a trajectory enters the closed recurrent class, all subsequent visits remain within that class. Consequently, although transient states may be visited repeatedly before entering the recurrent component, any further growth of the occupation vector is entirely determined by the recurrent dynamics.

Thus, the asymptotic behaviour of the occupation vectors is governed by the recurrent component of the support graph. The transient part of the support graph determines the finite prefix preceding entry into the recurrent class, whereas all asymptotic growth takes place within the recurrent component. This provides a combinatorial counterpart of the classical decomposition into transient and recurrent states.

\subsection{Simple Random Walk on $\mathbb{Z}$}
\label{subsec:random-walk-example}

Our third example shows that the occupation-pattern construction is not intrinsically tied to finite state spaces. Although the theoretical results of this paper are established for finite Markov chains, the construction itself depends only on finite admissible trajectories. Every finite trajectory in a locally finite directed support graph determines a finite word and hence a well-defined Parikh vector. The present example therefore illustrates how the construction naturally extends beyond the finite-state setting.

Consider the simple random walk on the integers. Its directed support graph is locally finite and is given by
\[
m\to m+1,\qquad
m\to m-1.
\]

Starting from the origin,
\[
A_n(0)=\{-n,-n+1,\ldots,n\},
\]
and therefore
\[
g_0(n)=2n+1.
\]
On the other hand,
\[
p_0(n)=2^n,
\]
since every trajectory of length $n$ is uniquely determined by its sequence of $n$ left/right choices.

Thus, reachability growth is linear, whereas trajectory growth is exponential. Occupation patterns provide a third level of description by recording how visits are distributed among the states reached by the walk.

For every fixed $n$, only finitely many states may be visited, namely those contained in $A_n(0)$. Consequently, every trajectory of length $n$ determines a finitely supported occupation vector, or equivalently, a Parikh vector indexed by the finite set $A_n(0)$.

At length two, the admissible trajectories are
\[
0\to1\to2,\qquad
0\to1\to0,\qquad
0\to(-1)\to0,\qquad
0\to(-1)\to(-2).
\]
Each trajectory determines a distinct occupation vector, so no Parikh compression occurs at this level.

For longer trajectories, however, different excursions may generate the same occupation pattern. For example,
\[
0\to1\to0\to1\to0\to(-1)\to0
\]
and
\[
0\to1\to0\to(-1)\to0\to1\to0
\]
have the same occupation vector when restricted to the visited states. In the coordinates corresponding to the states $0$, $1$ and $-1$, both trajectories determine the vector
\[
(4,2,1).
\]

Although the trajectories are distinct as ordered paths, they become indistinguishable after applying the Parikh map. This example shows that Parikh compression is not restricted to finite-state Markov chains. It also arises naturally in locally finite support graphs, suggesting that the occupation-pattern framework developed here extends well beyond the finite-state setting.

\section{Geometric and Algebraic Perspectives}
\label{sec:perspectives}

The constructions developed in the previous sections naturally suggest broader geometric and algebraic questions. In this final section we briefly discuss some of these perspectives and indicate possible directions for future research.

\subsection{Occupation Geometry}
\label{subsec:occupation-geometry}

The occupation patterns introduced in the previous sections admit a natural geometric interpretation. For a fixed initial state $i$, every admissible trajectory determines an occupation vector
\[
\alpha(\gamma)=(\alpha_1,\ldots,\alpha_m)\in\mathbb{N}^m,
\]
where $\alpha_r$ denotes the number of visits to state $r$.

Since every trajectory of length $n$ visits exactly $n+1$ states, counting multiplicities,
\[
\sum_{r=1}^m\alpha_r=n+1.
\]
Hence every occupation vector arising from a trajectory of length $n$ lies on the discrete simplex
\[
\Delta_n=
\left\{
x\in\mathbb{N}^m:
\sum_{r=1}^m x_r=n+1
\right\}.
\]

The finite Parikh image
\[
P_i^{(n)}=\Psi\bigl(L_i^{(n)}\bigr)
\]
may therefore be viewed as a finite subset of this simplex. Consequently,
\[
a_i(n)=|P_i^{(n)}|
\]
counts the lattice points of $\Delta_n$ realized by admissible trajectories of length $n$.

A schematic illustration is shown in Figure~\ref{fig:occupation-simplex}. For three states, the simplex is naturally two-dimensional. The figure illustrates the general geometric viewpoint: in higher dimensions, the support graph selects a finite subset of the corresponding lattice simplex consisting of the realizable occupation vectors.

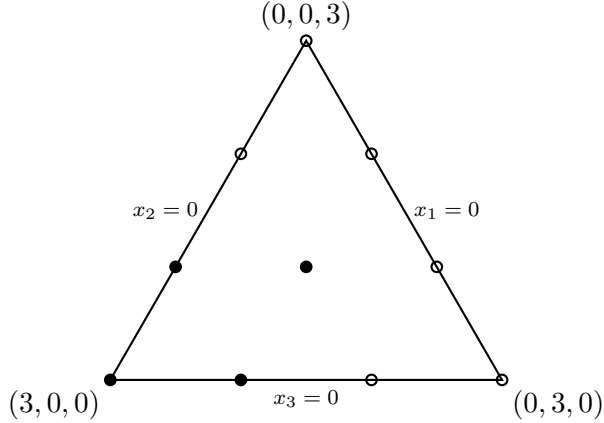
\begin{figure}[H]
\centering
\begin{tikzpicture}[scale=1.15]
\coordinate (A) at (0,0);
\coordinate (B) at (4.5,0);
\coordinate (C) at (2.25,3.897);

\draw[line width=0.8pt] (A) -- (B) -- (C) -- cycle;

\node[below left] at (A) {$(3,0,0)$};
\node[below right] at (B) {$(0,3,0)$};
\node[above] at (C) {$(0,0,3)$};

\node[left] at ($(A)!0.5!(C)$) {\scriptsize $x_2=0$};
\node[right] at ($(B)!0.5!(C)$) {\scriptsize $x_1=0$};
\node[below] at ($(A)!0.5!(B)$) {\scriptsize $x_3=0$};

\foreach \x/\y/\z in {3/0/0,2/1/0,1/2/0,0/3/0,2/0/1,1/1/1,0/2/1,1/0/2,0/1/2,0/0/3}{
  \coordinate (P\x\y\z) at ($\x/3*(A)+\y/3*(B)+\z/3*(C)$);
}

\foreach \x/\y/\z in {3/0/0,2/1/0,1/2/0,0/3/0,2/0/1,1/1/1,0/2/1,1/0/2,0/1/2,0/0/3}{
  \node[unrealized point] at (P\x\y\z) {};
}

\foreach \x/\y/\z in {3/0/0,2/1/0,2/0/1,1/1/1}{
  \node[simplex point] at (P\x\y\z) {};
}

\end{tikzpicture}
\caption{The discrete simplex $x_1+x_2+x_3=3$. Filled points represent a possible subset $P_i^{(2)}$ of realizable occupation vectors.}
\label{fig:occupation-simplex}
\end{figure}

From this viewpoint, the support graph imposes combinatorial constraints on which lattice points of the simplex can occur. Different support structures may therefore produce substantially different subsets of the same simplex. The transient part of the support graph determines the finite structure of the occupation geometry, whereas recurrent components determine its asymptotic geometry. In this sense, the occupation geometry may be viewed as a discrete geometric shadow of the support dynamics.

\subsection{Semilinear Geometry}
\label{subsec:semilinear-geometry}

The full Parikh image
\[
P_i=\bigcup_{n\ge0}P_i^{(n)}
\]
is semilinear by Proposition~\ref{prop:parikh-semilinear} \cite{esparza2011,parikh1966}. Consequently, it can be expressed as a finite union of linear sets
\[
v+\mathbb{N}u_1+\cdots+\mathbb{N}u_r.
\]

The semilinear structure also reflects the influence of directed cycles. Repeated traversals of a cycle add the same occupation increment each time, naturally producing arithmetic progressions in the Parikh image. This provides an intuitive geometric picture of why semilinear behaviour arises.

Accordingly, the large-scale geometry of $P_i$ is governed primarily by the cycle structure of the support graph. Transient regions contribute only finite prefixes, whereas recurrent classes generate the unbounded directions responsible for the asymptotic geometry of the occupation patterns. In this way, semilinearity provides a geometric description of how recurrent behaviour shapes the global structure of the set of occupation vectors.

\subsection{Algebraic Perspectives}
\label{subsec:algebraic-perspectives}

The occupation ideals introduced in Section~\ref{sec:occupation-ideals} provide a commutative-algebraic realization of the occupation patterns associated with admissible trajectories. For each fixed trajectory length $n$, the ideal
\[
J_n(i)
\]
encodes the finite set
\[
P_i^{(n)}
\]
of realizable occupation vectors. Its minimal generators are in one-to-one correspondence with these occupation patterns and therefore satisfy
\[
a_i(n)=\mu(J_n(i)).
\]

From this perspective, occupation ideals provide an algebraic representation of the combinatorial structure that remains after chronological information has been discarded. Unlike path-based constructions, which distinguish trajectories according to the order of their transitions, occupation ideals identify trajectories sharing the same occupation profile.

This representation places the construction within the framework of combinatorial commutative algebra and makes available a broad range of algebraic tools, including Hilbert functions, Hilbert series, Newton polyhedra, Betti numbers, projective dimension and Castelnuovo--Mumford regularity \cite{bruns1998,miller2005,stanley2012}. Understanding how these invariants reflect structural properties of the support graph and the geometry of occupation patterns appears to be a natural direction for further investigation.

Taken together, these constructions provide a combinatorial framework for the study of Markov support dynamics, opening the way to the systematic application of methods from formal language theory, discrete geometry and commutative algebra.

\section{Conclusion}
\label{sec:conclusion}

This paper introduces occupation patterns as a new combinatorial object associated with the support dynamics of discrete-time Markov chains. Starting from the directed support graph, admissible trajectories are viewed as a regular language whose Parikh image consists precisely of the occupation vectors realized by those trajectories. This construction connects Markov support graphs with formal language theory and combinatorial commutative algebra.

The resulting perspective naturally distinguishes three complementary levels of support complexity,
\[
\text{reachability}
\longrightarrow
\text{trajectory complexity}
\longrightarrow
\text{occupation-pattern complexity},
\]
quantified respectively by
\[
g_i(n),\qquad
p_i(n),\qquad
a_i(n).
\]
The occupation-pattern growth function
\[
a_i(n)=\#\Psi\bigl(L_i^{(n)}\bigr)=\mu(J_n(i))
\]
measures the distinct occupation patterns realized by admissible trajectories of length $n$, while the passage
\[
p_i(n)\longrightarrow a_i(n)
\]
quantifies the loss of chronological information induced by the Parikh map.

Beyond their combinatorial interpretation, occupation patterns admit complementary geometric and algebraic realizations. Geometrically, they appear as realizable lattice points in discrete simplices and form a semilinear subset of the corresponding lattice. Algebraically, they are represented by occupation ideals, making available methods from combinatorial commutative algebra for the study of Markov support dynamics.

The examples illustrate that occupation patterns capture structural features of the support graph that are not reflected solely by reachability or trajectory counts. In particular, the examples show that Parikh compression depends on the interaction between branching and directed cycles, while the semilinear geometry is governed by the recurrent structure of the support graph.

The framework developed here suggests several natural directions for further investigation, including the asymptotic classification of occupation-pattern growth, the geometric structure of Parikh images and the interpretation of algebraic invariants associated with occupation ideals. More broadly, it opens the way to applying methods from formal language theory, discrete geometry and commutative algebra to the study of Markov support dynamics.

\end{document}